\documentclass[reqno]{amsart}
\usepackage[utf8]{inputenc}
\usepackage{amssymb}
\usepackage{tikz}
\usepackage{euscript}
\usepackage{extarrows}
\usepackage{mathrsfs}
\usepackage{a4wide}
\usepackage{setspace}
\usepackage{graphicx}

\makeatletter
\@namedef{subjclassname@2010}{%
\textup{2010} Mathematics Subject Classification}
\makeatother

\newtheorem{thm}{Theorem}%[section]

\theoremstyle{remark}

\theoremstyle{definition}
\newtheorem{exa}[thm]{Example}

\DeclareMathOperator{\lin}{lin}

\DeclareMathOperator{\dzii}{{\mathsf{Chi}}}
\DeclareMathOperator{\koo}{{\mathsf{root}}}

\DeclareMathOperator{\paa}{{\mathsf{par}}}

\newcommand*{\cbb}{\mathbb C}

\newcommand*{\dz}[1]{{\EuScript D}(#1)}
\newcommand*{\dzi}[1]{\dzii(#1)}

\newcommand*{\hh}{\mathcal H}

\newcommand*{\is}[2]{\langle#1,#2\rangle}
\newcommand*{\jd}[1]{\EuScript N(#1)}

\newcommand*{\nbb}{\mathbb N}

\newcommand*{\pa}[1]{\paa(#1)}

\newcommand*{\rbb}{\mathbb R}

\newcommand*{\smalloplus}{\raise0pt\hbox{$\scriptscriptstyle \oplus$}}

\newcommand*{\slam}{S_{\boldsymbol \lambda}}
\newcommand*{\tcal}{{\mathscr T}}

\newcommand*{\zbb}{\mathbb Z}

\newcommand*{\lambdab}{\boldsymbol \lambda}

\newcommand*{\Le}{\leqslant}
\begin{document}
\setstretch{1.2}
\title[Complex symmetric of weighted shifts. II]{On complex symmetric weighted shifts. II}

\author[P.\ Budzy\'{n}ski]{Piotr Budzy\'{n}ski}
\address{Piotr Budzy\'{n}ski, Katedra Zastosowa\'{n} Matematyki, Uniwersytet Rolniczy w Krakowie, ul.\ Balicka 253c, 30-198 Kra\-k\'ow, Poland}
\email{piotr.budzynski@urk.edu.pl}

%\date{\today}

\begin{abstract}
Assorted weighted shifts over finite rooted directed trees are studied. Their complex symmetry is characterized.
\end{abstract}
\maketitle
\section{Introduction and preliminaries}
Suppose $\hh$ is a (complex separable) Hilbert space. An antilinear operator $C$ on $\hh$ satisfying $C^2=I$, where $I$ is the identity operator on $\hh$, and $\is{Cf}{Cg}=\is{g}{f}$ for all $f,g\in\hh$, is called a {\em conjugation}. A bounded operator $T$ on $\hh$ is {\em $C$-symmetric} if $T= CT^*C$.

Complex symmetric operators gained a lot of interest in the past following the research done by Garcia, Putinar, and Wogen (see \cite{g-jmaa-2007, g-otaa-2008, g-p-tams-2006, g-p-tams-2007, g-p-tmj-2008, g-w-jfa-2009, g-w-tams-2010}). Among the many papers devoted to studying the properties of complex symmetric operators was a one by Zhu and Li (see \cite{z-l-tams-2013}) characterizing (classical) bounded unilateral and bilateral complex symmetric weighted shifts. The characterization showed the importance of truncated weighted shifts. As it turns out, they are the building blocks for bounded complex symmetric weighted shifts. The same happens to be true in the unbounded case under some additional assumptions (see \cite{b-b-2023})

Classical weighted shifts have natural and important generalizations -- weighted shifts on directed trees. The class composed of the latter operators was introduced by Jabłoński, Jung, and Stochel (see \cite{j-j-s}). An initial motivation for introducing this class comes from the research on adjacency operators done Fujii, Sasaoka, and Watatani (see \cite{f-s-w-mj-1989}).

Let $\tcal=(V,E)$ be a directed tree ($V$ and $E$ stand for the sets of vertices and edges of $\tcal$, respectively). Set $\dzi u = \{v\in V\colon (u,v)\in E\}$ for $u \in V$. Denote by $\paa$ the partial function from $V$ to $V$ which assigns to a vertex $u\in V$ its parent $\pa{u}$ (i.e.\ a unique $v \in V$ such that $(v,u)\in E$). A vertex $u \in V$ is called a {\em root} of $\tcal$ if $u$ has no parent. A root is unique (provided it exists); we denote it by $\koo$. Set $V^\circ=V \setminus \{\koo\}$ if $\tcal$ has a root and $V^\circ=V$ otherwise. We say that $u \in V$ is a {\em branching vertex} of $V$, and write $u \in V_{\prec}$, if $\dzi{u}$ consists of at least two vertices. 

Assume $\lambdab=\{\lambda_v\}_{v \in V^{\circ}} \subseteq \cbb$ satisfies $\sup_{v\in V}\sum_{u\in\dzii{v}}|\lambda_u|^2<\infty$. Then the following formula
\begin{align}\label{xx}
(\slam f)(v)=
   \begin{cases}
\lambda_v \cdot f\big(\pa v\big) & \text{ if } v\in V^\circ,
   \\
0 & \text{ if } v=\koo,
   \end{cases}
\end{align}
defines a bounded operator on $\ell^2(V)$ (as usual, $\ell^2(V)$ is the Hilbert space of square summable complex functions on $V$ with standard inner product). We call it a {\em weighted shift on} $\tcal$ with weights $\lambdab$. It is known that
\begin{align*}
(\slam^* f)(v)=\sum_{u\in\dzii{v}}\overline{\lambda_u}f(u),\quad v\in V.
\end{align*}
This and other necessary facts concerning weighted shifts on directed trees can be found in a monograph \cite{j-j-s}. 

Studying weighted shifts on directed trees led to numerous interesting results (see e.g., \cite{b-j-j-s-jmaa-2013, b-j-j-s-aim-2017, b-j-j-s, j-j-s-jfa-2012, j-j-s-pams-2014, p-laa-2023}). It is worth mentioning that the research showed that the structure of the underlying tree is rigid in the case of selfadjoint and normal weighted shifts on directed trees, and having these properties reduces a weighted shift on a directed tree to a classical one (see \cite{j-j-s-caot-2013})

Having in mind all the mentioned above ii seems natural to address a problem of complex symmetry for directed tree generalizations of the classical truncated weighted shifts. In this short note we focus on two particular types operators that fall in this category.

Viewed as a weighted shift on a directed tree, a truncated weighted shift operator is associated with a finite rooted directed tree without branching vertices, or in other words, having just one branch (see Figure 1 below).
\vspace{3ex}
\begin{center}
\begin{tikzpicture}[scale = .6, transform shape]
\tikzstyle{every node} = [circle,fill=gray!30]

\node (e0) at (-3,1) {}; 
\node (e1) at (-1,1) {};
\node (e2) at (1,1) {}; 
\node (e3) at (3,1) {}; 
\node (e4) at (5,1) {}; 
\node (e6) at (8,1) {};

\draw[->] (e0) --(e1) node[pos=0.5,above = 0pt,fill=none] {}; 
\draw[->] (e1) --(e2) node[pos=0.5, below=0pt,fill=none] {};
\draw[->] (e2) --(e3) node[pos=0.5,below=0pt,fill=none] {};
\draw[->] (e3) --(e4) node[pos=0.5,above = 0pt,fill=none] {}; 
\draw[dashed,->] (e4) -- (e6) ;
\end{tikzpicture}
\end{center}
\begin{center}
Figure 1
\end{center}
\vspace{1ex}
The first generalizing operator, in some sense a simpler one, is associated with a rooted directed tree with one branching vertex and two equally long branches, meaning it arizes from adding another branch. We will denote such trees by $\tcal_{\kappa,\theta}$ (see Figure 2 below).
\vspace{3ex}
\begin{center}
\begin{tikzpicture}[scale = .6, transform shape]
\tikzstyle{every node} = [circle,fill=gray!30]

\node (e1kappa)[font=\footnotesize] at (-4,0) {}; 
\node (ekap+1)[font=\footnotesize] at (-2,0) {};
\node (e-1)[font=\footnotesize] at (1,0) {}; 
\node (e10) at (3,0) {}; 
\node (e11) at (5,1.5) {}; 
\node (e12) at (7,1.5) {}; 
\node (e13) at (10,1.5) {};
\node (e21) at (5,-1.5) {}; 
\node (e22) at (7,-1.5) {}; 
\node (e23) at (10,-1.5) {};

\draw[->] (e10) --(e11) node[pos=0.5,above = 0pt,fill=none] {}; 
\draw[->] (e10) --(e21) node[pos=0.5, below=0pt,fill=none] {};
\draw[->] (e11) --(e12) node[pos=0.5,above = 0pt,fill=none] {}; 
\draw[->] (e21) --(e22) node[pos=0.5,below=0pt,fill=none] {};
\draw[dashed,->] (e12) --(e13) node[pos=0.5,above = 0pt,fill=none] {}; 
\draw[dashed,->] (e22) --(e23) node[pos=0.5,below = 0pt,fill=none] {};
\draw[ ->] (e-1) --(e10) node[pos=0.5,above = 0pt,fill=none] {}; 
\draw[->] (e1kappa) -- (ekap+1);
\draw[dashed,->] (ekap+1) -- (e-1) ;
\end{tikzpicture}
\end{center}
\begin{center}
Figure 2
\end{center}
\vspace{1ex}
The second generalizing operator arises from looking at the graph associated to a truncated weighted shift weighted shifts operators with valency (the number of outgoing edges from a given vertex) in mind. The valency of that graph at all but one vertices equals to 1 and at the one that has valency not equal to 1 it is equal to 0 (such a vertex is called a leaf). Increasing valency to 2 at all the vertices except leaves we get a (complete) binary tree and this is the associated graph of the operator. Such trees will be denoted by $\tcal^2_\kappa$, where $\kappa$ is the depth of the tree (see Figure 3 below for the a binary of depth 3).
\vspace{3ex}
\begin{center}
\begin{tikzpicture}[scale = .6, transform shape, line cap=round,line join=round,x=1.2cm,y=1cm]
   \tikzstyle{every node} = [circle, fill=gray!30]
\node (v1)[font=\footnotesize]at (1,0){};

\node (v11)[font=\footnotesize] at (3.5,1.125){};
\node (v12)[font=\footnotesize] at (3.5,-1.125){};

\node (v111)[font=\footnotesize] at (6,3){};
\node (v112)[font=\footnotesize] at (6,1.25){};
\node (v121)[font=\footnotesize] at (6,-1.25){};
\node (v122)[font=\footnotesize] at (6,-3){};

\node (v1111)[font=\footnotesize] at (9,3.75){};
\node (v1112)[font=\footnotesize] at (9,2.75){};
\node (v1121)[font=\footnotesize] at (9,1.75){};
\node (v1122)[font=\footnotesize] at (9,0.75){};
\node (v1211)[font=\footnotesize] at (9,-1.75){};
\node (v1212)[font=\footnotesize] at (9,-0.75){};
\node (v1221)[font=\footnotesize] at (9,-3.75){};
\node (v1222)[font=\footnotesize] at (9,-2.75){};

\draw[->] (v1) --(v11) node[pos=0.6,below = -5pt,fill=none, font = \footnotesize] {};
\draw[->] (v1) --(v12) node[pos=0.6,below =  -3pt,fill=none, font = \footnotesize] {};
\draw[->] (v11) --(v111) node[pos=0.6,below = -5pt,fill=none, font = \footnotesize] {};
\draw[->] (v11) --(v112) node[pos=0.6,below =  -3pt,fill=none, font = \footnotesize] {};
\draw[->] (v12) --(v121) node[pos=0.6,below = -5pt,fill=none, font = \footnotesize] {};
\draw[->] (v12) --(v122) node[pos=0.6,below = -3pt,fill=none, font = \footnotesize] {};

\draw[->] (v111) --(v1111) node[pos=0.6,below = -3pt,fill=none, font = \footnotesize] {};
\draw[->] (v111) --(v1112) node[pos=0.6,below = -3pt,fill=none, font = \footnotesize] {};
\draw[->] (v112) --(v1121) node[pos=0.6,below = -3pt,fill=none, font = \footnotesize] {};
\draw[->] (v112) --(v1122) node[pos=0.6,below = -3pt,fill=none, font = \footnotesize] {};
\draw[->] (v121) --(v1211) node[pos=0.6,below = -3pt,fill=none, font = \footnotesize] {};
\draw[->] (v121) --(v1212) node[pos=0.6,below = -3pt,fill=none, font = \footnotesize] {};
\draw[->] (v122) --(v1221) node[pos=0.6,below = -3pt,fill=none, font = \footnotesize] {};
\draw[->] (v122) --(v1222) node[pos=0.6,below = -3pt,fill=none, font = \footnotesize] {};
   
\end{tikzpicture}
\end{center}
   \begin{center}
Figure 3
   \end{center}
\vspace{1ex}
With all the above in mind characterizing complex symmetry of the weighted shifts over directed trees that were mentioned before seems natural and interesting. We do it under  two additional assumptions in Theorems \ref{kaszuby1} and \ref{kaszuby2}. Namely, we assume that all the weights are nonzero and constant across the generations. Both of them make the weighted shifts we consider the most simple generalizations of the classical truncated weighted shift. Also, they make the proof elementary and quite simple to follow. Further, we provide very simple but informative examples concerning the subject.

This research is a continuation of \cite{b-b-2023} where unbounded complex symmetric  unilateral and bilateral weighted shifts were studied and in which generalizations of  results from \cite{z-l-tams-2013} were given. In the last part of paper we deal briefly with unbounded operators but the arguments don't use any methods outside the linear algebra and very basic operator theory.
\section{Two examples}
We begin with two elementary examples showing that complex symmetry of a weighted shift on a directed tree does not depend purely on its weights but also on the structure of the underlying directed tree even if the structure is very simple. 
\begin{exa}\label{91}
Let $\tcal=(V,E)$, where 
$$V=\{0, (1,1), (2,1), (2,2)\}, \quad E=\{(0, (1,1)), (0,(2,1)), ((2,1),(2,2))\}.$$
Let $\lambdab=\{\lambda_{1,1}, \lambda_{2,1}, \lambda_{2,2}\}\subseteq\cbb\setminus\{0\}$ with 
\begin{align*}
\lambda_{1,1} =\lambda_{2,1}\quad \text{and}\quad \lambda_{2,2}=\sqrt2\lambda_{2,1}.
\end{align*}
Finally, let $\slam$ be a weighted shift on $\tcal$ with weights $\lambdab$. Using \eqref{xx} we get
\begin{align}\label{x}
\slam e_0&=\lambda_{1,1}e_{1,1} + \lambda_{2,1}e_{2,1}=\lambda_{1,1}(e_{1,1} + e_{2,1}),\notag\\ 
\slam e_{2,1}&=\lambda_{2,2} e_{2,2}={\sqrt{2}}\lambda_{1,1}e_{2,2},\\ 
\slam e_{1,1}&=\slam e_{2,1}=0.\notag
\end{align}
Figure 4 below is a graphical representation of the above -- each node in the graph is related to a unique vector in the orthonormal basis $\{e_v\}_{v\in V}$ and for a fixed node $v\in V$, the outgoing arrows indicate what vectors from the orthonormal basis and what weights are used in a linear combination that serves as the value of $\slam$ on $e_v$.
\vspace{3ex}
\begin{center}
\begin{tikzpicture}[scale = .6, transform shape]
\tikzstyle{every node} = [circle,fill=gray!30]
\node (e0) at (1,0) {}; 
\node (e11) at (4,1) {$\lambda_{1,1}$}; 
\node (e21) at (4,-1) {$\lambda_{1,1}$}; 
\node (e22) at (7,-1) {${\sqrt{2}}\lambda_{1,1}$}; 

\draw[->] (e0) --(e11) node[pos=0.5,above = 0pt,fill=none] {}; 
\draw[->] (e0) --(e21) node[pos=0.5, below=0pt,fill=none] {};
\draw[->] (e21) --(e22) node[pos=0.5,below=0pt,fill=none] {}; 
\end{tikzpicture}
\end{center}
   \begin{center}
Figure 4
   \end{center}
Moreover, we have
\begin{align}\label{xxx}
\slam^*e_0=0,\
\slam^*e_{1,1}=\overline{\lambda_{1,1}}e_0,\
\slam^*e_{2,1}=\overline{\lambda_{2,1}}e_0,\
\slam^*e_{2,2}=\overline{\lambda_{2,2}}e_{2,1}.
\end{align}
Let $C$ be a unique conjugation on $\ell^2(V)$ satisfying
\begin{align}\label{kupa}
Ce_0=e_{2,2},\ Ce_{2,2}= e_{0},\ C e_{1,1}=\frac{1}{\sqrt2}\big(e_{2,1}-e_{1,1}\big),\ C e_{2,1}=\frac{1}{\sqrt2}\big(e_{1,1}+e_{2,1}\big).
\end{align}
Combining \eqref{x}, \eqref{xxx} and \eqref{kupa} one can easily show that  $\slam$ is $C$-symmetric.
\end{exa}
Expanding the trunk of the tree from Example \ref{91} leads to non complex symmetric weighted shift.
\begin{exa}
Let $\tcal=(E,V)$, where 
\begin{align*}
V=\{-1, 0, (1,1), (2,1), (2,2)\}\quad \text{and} \quad E=\{(-1,0),(0,(1,1)),(0,(21)), ((2,1),(2,2))\}.
\end{align*}
Given any $\lambdab=\{\lambda_0, \lambda_{1,1}, \lambda_{2,1}, \lambda_{2,2}\}\subseteq \cbb\setminus\{0\}$ we consider $\slam$ on $\ell^2(V)$ (see Figure 5).
\vspace{3ex}
\begin{center}
\begin{tikzpicture}[scale = .6, transform shape]
\tikzstyle{every node} = [circle,fill=gray!30]
\node (e-1)[font=\footnotesize] at (0,0) {}; 
\node (e0) at (2,0) {$\lambda_0$}; 
\node (e11) at (4,1) {$\lambda_{1,1}$}; 
\node (e21) at (4,-1) {$\lambda_{2,1}$}; 
\node (e22) at (6,-1) {$\lambda_{2,2}$}; 

\draw[->] (e0) --(e11) node[pos=0.5,above = 0pt,fill=none] {}; 
\draw[->] (e0) --(e21) node[pos=0.5, below=0pt,fill=none] {};
\draw[->] (e21) --(e22) node[pos=0.5,below=0pt,fill=none] {};
\draw[ ->] (e-1) --(e0) node[pos=0.5,above = 0pt,fill=none] {}; 
\end{tikzpicture}
\end{center}
\begin{center}
Figure 5
\end{center}

Elementary calculations give
\begin{align*}
\slam^2e_{-1}=\lambda_0\lambda_{1,1}e_{1,1}+\lambda_0\lambda_{2,1}e_{2,1},\ 
\slam^2e_{1,1}=\slam^2 e_{2,1}=\slam^2e_{2,2}=0,\ 
\slam^2e_{0}=\lambda_{2,1}\lambda_{2,2}e_{2,2}
\end{align*}
and
\begin{align*}
\slam^{*2}e_{-1}=\slam^{*2}e_0=\slam^{*2}e_{1,1}=\slam^{*2}e_{2,1}=0,\ \slam^{*2}e_{2,2}=\overline{\lambda_{2,1}\lambda_{2,2}}e_0.
\end{align*}
Assuming $\lambda_{1,1}\neq-\lambda_{2,1}$ we get a weighted shift that cannot be complex symmetric because the kernel $\jd{\slam^{*2}}$ is four dimensional while the kernel $\jd{\slam^2}$ is three dimensional.
\end{exa}
\section{Complex symmetry over $\tcal_{\kappa,\theta}$ and  $\tcal^2_\kappa$}
In this section we focus on complex symmetry of weighted shifts over directed trees $\tcal_{\kappa,\theta}$ and  $\tcal^2_\kappa$. In both cases we will assume:
\begin{itemize}
\item the weights are constant across the generations of the tree (which is trivially satisfied in the classical case),
\item the weights are nonzero.
\end{itemize}
 
First, we formally define $\tcal_{\kappa,\theta}$. Let $\kappa\in \zbb_+$ and $\theta \in\nbb$  with $\nbb=\{1,2,3 \ldots\}$ and $\zbb_+=\nbb\cup\{0\}$. Set
\allowdisplaybreaks
\begin{align*}
V_{\kappa,\theta} & = \big\{-k\colon k\in J_\kappa\big\} \cup \{0\} \cup \big\{(i,j)\colon i\in \{1,2\},\, j\in J_\theta\big\}
\end{align*}
and
\begin{align*}
E_{\kappa, \theta} =\big\{(-k,-k+1)\colon k\in J_\kappa\big\}&\cup \big\{(0, (1,1)), (0,(2,1))\big\}\\
&\cup\big\{((i,j),(i,j+1))\colon i\in\{1,2\},\ j\in J_{\theta-1}\big\},
\end{align*}
where $J_n = \{k \in \nbb\colon k\Le n\}$ for $n \in \zbb_+$. 
\begin{thm}\label{kaszuby1}
Let $\kappa\in \zbb_+$ and $\theta \in\nbb$. Let $\lambdab=\{\lambda_v\}_{V_{\kappa,\theta}^\circ}\subseteq \cbb\setminus\{0\}$ satisfy $\lambda_{1,j}=\lambda_{2,j}=:\lambda_j$, $j\in J_\theta$. Then the weighted shift operator $\slam$ on a directed tree $\tcal_{\kappa,\theta}=(V_{\kappa,\theta}, E_{\kappa, \theta})$ is complex symmetric if and only if the following conditions are satisfied:
\begin{itemize}
\item[(i)] $|\lambda_{1+j}|=|\lambda_{\theta+1-j}|$ for every $j\in J_{\theta-1}$,
\item[(ii)] for $\theta-\kappa=1$, $|\lambda_{-\kappa+j}|=|\lambda_{\theta-j+1}|$ for every $j\in J_{\kappa+\theta}$
\item[(iii)] for $\theta-\kappa\neq1$,
$\sqrt{2}|\lambda_{1}|=|\lambda_{\theta-\kappa}|$ and $|\lambda_{-\kappa+j}|=|\lambda_{\theta-j+1}|$ for every $j\in \big(J_{\kappa+\theta}\setminus\{\kappa\}\big)$.
\end{itemize}
\end{thm}
\begin{proof}
Let us assume that $\lambdab=\{\lambda_v\}_{V^\circ}\subseteq (0,+\infty)$ (this can be done by \cite[Theorem 3.2.1]{j-j-s}).

We see that
\begin{align*}
\slam e_{-l} &= \lambda_{-l+1}e_{-l+1},\quad l\in J_\kappa,\quad 
&&\slam e_{0} = \lambda_{1}e_{1,1}+\lambda_{1}e_{2,1},\\
\slam e_{i,j}&=\lambda_{j+1}e_{i,j+1},\quad  j\in J_{\theta-1},\quad
&&\slam e_{i,\theta}=0,\quad i\in\{1,2\}.
\end{align*}

We put 
\begin{align*}
\hh_j^+=\bigvee\{e_{1,j}+e_{2,j}\},\quad \hh_j^-=\bigvee\{e_{1,j}-e_{2,j}\},\quad \hh_j=\hh_j^+\oplus\hh_j^-, \quad j\in J_\theta
\end{align*}
and
\begin{align*}
\hh_{-j}^+=\hh_{-j}=\cbb e_{-j},\quad j\in J_\kappa\cup\{0\}.
\end{align*}
Clearly, we have $\ell^2(V)=\bigoplus_{j\in J_\kappa} \hh_{-j}\oplus \hh_0\oplus \bigoplus_{j\in J_\theta} \hh_j$. 

It is a matter of simple computations that
\begin{align}\label{tere1}
\jd{\slam^{m}}= \hh_{\theta}\oplus \ldots \oplus \hh_{\theta-m+1},\quad m\in\nbb,
\end{align}
and
\begin{align}\label{tere2}
\jd{\slam^{*m}}=\hh_{-\kappa}^+\oplus\ldots \oplus\hh_{-\kappa+m-1}^+\oplus \hh_1^-\oplus \ldots \oplus\hh_m^-,\quad m\in\nbb,
\end{align}
with $\hh_j=\{0\}$ for $j<-\kappa$ and $\hh_l^+=\hh_l^-=\{0\}$ for $l>\theta$.

Assume that $\slam$ is complex symmetric with respect to a conjugation $C$. Using $C\jd{\slam^{*k}}=\jd{\slam^k}$ for $k\in\nbb$, we deduce from \eqref{tere1} and \eqref{tere2} that
\begin{align}\label{tere3}
C \big(\hh_{-\kappa+m}^+\oplus \hh_{m+1}^{-}\big)=\hh_{\theta-m}^+\oplus \hh_{\theta-m}^-,\quad m\in\zbb_+,
\end{align}
with $\hh_l^+=\{0\}$ for $l<-\kappa$ and $\hh_j^{-}=\{0\}$ for $j\leq 0$. Substituting $m=\kappa+\theta$ into \eqref{tere3} we get
\begin{align*}
C\hh_{\theta}^+=\hh_{-\kappa}^+\text{ and }C\hh_{\theta}^-=\hh_1^-.
\end{align*}
Substituting $m=1$ into \eqref{tere3}, we get
\begin{align}\label{bah1}
C(\hh_{-\kappa+1}^+\oplus \hh_2^-)&=\hh_{\theta-1}^+\oplus \hh_{\theta-1}^-.
\end{align}
Substituting $m=\kappa+\theta-1$ into \eqref{tere3} and using $C^2=I$, we get
\begin{align}\label{bah2}
C(\hh_{-\kappa+1}^+\oplus \hh_{-\kappa+1}^-)=\hh_{\theta-1}^+\oplus \hh_{\theta+\kappa}^-
\end{align}
Comparing \eqref{bah1} and \eqref{bah2} we get
\begin{align*}
C\hh_{\theta-1}^+=\hh_{-\kappa+1}^+\text{ and }C\hh_{2}^-=\hh_{\theta-1}^-.
\end{align*}
In general, we have
\begin{align}\label{tere4}
C\hh_{\theta-l}^+=\hh_{-\kappa+l}^+ \text{ and }C\hh_{1+l}^-=\hh_{\theta-l}^- \quad \text{for $l\in\zbb_+$.}
\end{align}

In the remainder of the proof we will show all the relations between the weights. For this we define another orthonormal basis of $\ell^2(V_{\kappa,\theta})$
\begin{align*}
f_{-j}&=e_{-j},\quad j\in J_\kappa\cup\{0\},\\
f_j&=\frac{1}{\sqrt2}(e_{1,j}+e_{2,j}), \quad j\in J_{\theta},\\
g_j&=\frac{1}{\sqrt2}(e_{1,j}-e_{2,j}),\quad j\in J_{\theta}.
\end{align*}
By \eqref{tere4}, we have
\begin{align*}%\label{jojki2}
Cg_{1+j}=\delta_{j} g_{\theta-j},\quad j\in J_{\theta-1}\cup\{0\}
\end{align*}
with $\delta_{j}$ of modulus 1. Therefore, we have
\begin{align}\label{jo1}
\slam C g_{1+j}=\delta_j \slam g_{\theta-j}=\delta_j\lambda_{\theta-j+1}g_{\theta-j+1}, \quad j\in J_{\theta-1}. 
\end{align}
Comparing the above with
\begin{align}\label{jo2}
C\slam^* g_{1+j}=\lambda_{1+j}Cg_{j}=\lambda_{1+j}\delta_{j-1} g_{\theta-j+1}, \quad j\in J_{\theta-1},
\end{align}
implies (i).

Using \eqref{tere4} we see that
\begin{align*}
C f_{-\kappa+j}=\gamma_j f_{\theta-j}, \quad j\in J_{\theta+\kappa}\cup\{0\}
\end{align*}
with $\gamma_j$ of modulus 1. Thus we have
\begin{align}\label{jo3}
\slam Cf_{-\kappa+j}=\gamma_j\slam f_{\theta-j}=\gamma_j\mu_j \lambda_{\theta-j+1}f_{\theta-j+1},\quad j\in J_{\theta+\kappa},
\end{align}
with $\mu_j\in\{1,\sqrt{2}\}$ such that $\mu_j=\sqrt2$ if and only if $j=\theta$, and
\begin{align}\label{jo4}
C\slam^* f_{-\kappa+j}=\nu_j\lambda_{-\kappa+j} Cf_{-\kappa+j-1}=\nu_j\lambda_{-\kappa+j} \gamma_{j-1}f_{\theta-j+1},\quad j\in J_{\theta+\kappa}
\end{align}
with $\nu_j\in\{1,\sqrt2\}$ such that $\nu_j=\sqrt{2}$ if and only if $\kappa+1=j$. Comparing these two we get
\begin{align*}
\mu_j |\lambda_{\theta-j+1}|=\nu_j|\lambda_{-\kappa+j}|,\quad j\in J_{\theta+\kappa}.
\end{align*}
Using the above one can deduce (ii) and (iii). This proves the necessity part.

Now we show that (i)-(iii) are sufficient. Note that by (i), (ii) and (iii) 
\begin{align*}
\lambda_{1+j}\delta_{j-1}&=\delta_j\lambda_{\theta-j+1}, \quad j\in J_{\theta-1},\\
\nu_j\lambda_{-\kappa+j} \gamma_{j-1}&=
\gamma_j\mu_j \lambda_{\theta-j+1},\quad j\in J_{\theta+\kappa},
\end{align*}
with $\delta_j$'s and $\gamma_j$'s of modulus 1. These together with \eqref{tere4} can be used to define $C$ on $\hh$ uniquely. Moreover, in view \eqref{jo1}-\eqref{jo4}, $\slam$ is $C$ symmetric, which completes the proof.
\end{proof}
Now we address the problem of complex symmetry of weighted shift over a binary tree $\tcal_\kappa^2$. Formally the tree can be defined as follows. Let $\kappa\in \nbb$ be greater or equal to $2$. Set
\begin{align*}
V_\kappa^2=\{(k,l)\colon k \in J_{\kappa}\cup\{0\} \text{ and } l \in J_{2^k}\},    
\end{align*}
and
\begin{align*}
E_\kappa^2=\big\{\big((k,l),(k+1,2^l-1)\big)\colon k \in J_{\kappa}\cup\{0\},\ l \in J_{2^k}\big\}&\\ 
\cup \big\{\big((k,l),(k+1,2^l)\big)\colon k \in J_{\kappa}\cup\{0\},\ l \in J_{2^k}\big\}&,    
\end{align*}
\begin{thm}\label{kaszuby2}
Let $\kappa\in \nbb$ be greater or equal to $2$. Let $\lambdab=\{\lambda_v\}_{(V_{\kappa}^2)^\circ}\subseteq \cbb\setminus\{0\}$ satisfy $\lambda_{k,j}=:\lambda_k$, $j\in J_{2^k}$. Then the weighted shift operator $\slam$ on a directed tree $\tcal_{\kappa}^2=(V_{\kappa}^2, E_{\kappa}^2)$ is complex symmetric if and only if the following condition is satisfied
\begin{align}\label{rita2}
2|\lambda_{l+1}|=|\lambda_{\kappa-l}|,\quad l\in J_{\kappa}\cup\{0\}.
\end{align}
\end{thm}
\begin{proof}
We assume without loosing generality that all $\lambda_v$'s are positive. For $k\in J_\kappa\cup\{0\}$ we set $f_k=\sum_{l\in J_{2^k}}e_{k,l}$. Assuming that $\slam$ is $C$-symmetric with a conjugation $C$ and using $C\jd{\slam^\kappa}=\jd{\slam^{*\kappa}}$ we get $C(\bigvee\{ f_\kappa\})=\bigvee\{ f_0\}$. Furthermore, we see that
\begin{align*}
\lambda_1 f_1=\slam f_0=C\slam^* Cf_0 =\alpha(\kappa) C\slam^* f_\kappa=\alpha(\kappa) \lambda_\kappa Cf_{\kappa-1}
\end{align*}
with $\alpha(\kappa)\in \cbb$ such that $|\alpha(\kappa)| \,\|f_\kappa\|=\|f_0\|$. Thus $|\alpha(\kappa)|=\frac{1}{\sqrt{2^\kappa}}$. Moreover, $C(\bigvee\{ f_{\kappa-1}\})=\bigvee\{ f_1\}$ and $2 |\lambda_1|= |\lambda_{\kappa}|$. Applying the same argument more times we deduce that
\begin{align*}
C \big(\cbb f_{\kappa-l}\big)=\cbb f_l,\quad l\in J_{\kappa}\cup\{0\}
\end{align*}
which yields \eqref{rita2} and shows that it is necessary for $\slam$ to be complex symmetric. 

It is a matter of further elementary calculations that \eqref{rita2} is also sufficient for complex symmetry of $\slam$ (one argues in a similar way as in the proof of Theorem \ref{kaszuby1}).
\end{proof}

\section{Infinite trees - two examples}
A truncated weighted shift is associated to a directed tree having a finite depth. Clearly, a directed tree of finite depth need not to be a finite one, so a natural question arises if an infinite directed tree of finite depth admits a complex symmetric weighted shift. Below we show that it is possible by providing an example of a complex symmetric weighted shift on an infinite broom-like tree of depth 1.
\begin{exa}
Let $V=\zbb_+$ and $E=\{(0,j)\colon j\in\nbb\}$. Let $\lambdab=\{\lambda_v\}_{v\in V^\circ}\subseteq (1,\infty)$ satisfy $\sum_{i\in\nbb}\lambda_i^2<\infty$ and other conditions which would be specified later. Let $\slam$ be a weighted shift on a directed tree $\tcal=(E,V)$ induced by $\lambdab$. It is clear that $\slam$ is a densely defined operator such  $\jd{\slam}=\bigvee\{e_j\colon j\in\nbb\}$. Moreover, by \cite{j-j-s}, $\jd{\slam^*}^\perp=\cbb f_0$ with $f_0=\frac{\sum_{i\in\nbb}\lambda_ie_i}{\big(\sum_{i\in\nbb}\lambda_i^2\big)^{1/2}}$. 

Suppose for a moment that $\slam$ is complex symmetric with respect to a conjugation $C$. Then $C\jd{\slam}^\perp=\jd{\slam^*}^\perp$ and thus $Ce_0=\alpha f_0$ for some complex $\alpha$ of modulus 1. Obviously this implies that $\slam C$ equals $C\slam^*$ on $\bigvee\{e_0, f_0\}$. Set 
\begin{align}\label{tjp0}
g_i=Ce_i, \quad i\in\nbb.
\end{align}
Since $C\slam^*e_i=\lambda_i Ce_0=\lambda_i\alpha f_0$, we see that $\slam g_i=\alpha \lambda_i f_0$ for every $i\in\nbb$. Thus
\begin{align*}
\slam\hat{g}_i=f_0,\quad i\in\nbb,
\end{align*}
with $\hat{g}_i=\frac{g_i}{\alpha\lambda_i}$, which implies
\begin{align}\label{tjp0.5}
\hat{g}_i=e_0+h_i,\quad i\in\nbb
\end{align}
with some $h_i\in\jd{\slam}\ominus\cbb f_0$. This yields $\frac{1}{\lambda_i^2}=\|\hat{g}_i\|^2=1+\|h_i\|^2$ for every $i\in\nbb$ and so
\begin{align}\label{tjp1}
 \|h_i\|^2=\frac{1-\lambda_i^2}{\lambda_i^2},\quad i\in\nbb.
\end{align}
We note here the assumption of $\lambda_i>1$. Moreover, since $0=\langle \hat{g}_i,\hat{g}_j\rangle=1+\langle h_i, h_j\rangle$ for $i\neq j$, we have
\begin{align}\label{tjp2}
\is{h_i}{h_j}=-1,\quad i,j \in\nbb \text{ such that }i\neq j.
\end{align}
All the above means that existence of a sequence $\{h_i\}_{i\in\nbb}\subseteq \jd{\slam}$ of vectors satisfying \eqref{tjp1} and \eqref{tjp2} is necessary for $\slam$ to be complex sefladjoint. It is also clear that given a sequence $\{h_i\}_{i\in\nbb}\subseteq \jd{\slam}$  satisfying \eqref{tjp1} and \eqref{tjp2} one can define conjugation $C$ via \eqref{tjp0} and \eqref{tjp0.5} such that $C\slam^*$ and $\slam C$ equal on $\lin\{e_i\colon i\in \nbb\}$, the linear span of $\{e_i\colon i\in \nbb\}$. Below we address the problem of existence of such a sequence $\{h_i\}_{i\in\nbb}$ using induction argument.

Let $\{f_j\colon j\in\nbb\}$ be an orthonormal basis of $\jd{\slam^*}\ominus\bigvee\{e_0, f_0\}$. Then $\{f_j\colon j\in\nbb\}\subseteq \jd{\slam}$. Let $h_1\in\cbb f_1$ satisfy \eqref{tjp1}. Set 
\begin{align*}
h_2=t_{21}h_1+s_{2}f_2
\end{align*}
with $t_{21}=\frac{-1}{\|h_1\|^2}=\frac{\lambda_1^2-1}{\lambda_1^2}$ and $s_{2}^2=\frac{1-\lambda_2^2}{\lambda_2^2}-\frac{\lambda_1^2}{1-\lambda_1^2}$. Clearly, $s_{2}$ is properly defined whenever $(1-\lambda_1^2)(1-\lambda_2^2)>\lambda_1^2\lambda_2^2$. In that case, one can easily verify that vectors in $\{h_1,h_2\}$ satisfy equalities in \eqref{tjp1} and \eqref{tjp2}. Now we assume that for $n\in\nbb$ there exists $\{h_1,\ldots , h_n\}\subseteq \lin\{f_1, \ldots, f_n\}\subseteq\jd{\slam}\ominus\cbb f_0$ such the equalities in \eqref{tjp1} and \eqref{tjp2} are satisfied. Put
\begin{align*}
h_{n+1}=t_{n+1,1}h_1+\ldots +t_{n+1,n}h_n+s_{n+1}f_{n+1}
\end{align*}
with some unspecified $t_{n+1,j}$ for $j\in J_{n}$ that, in view of \eqref{tjp1} and \eqref{tjp2}, have to satisfied
\begin{align}\label{neon1}
\begin{cases}
    1=&t_{n+1,1}\frac{\lambda_1^2-1}{\lambda_1^2}+t_{n+1,2}+\ldots +t_{n+1,n},\\
    1=&t_{n+1,1}+t_{n+1,2}\frac{\lambda_2^2-1}{\lambda_2^2}+\ldots +t_{n+1,n},\\
    &\ldots\\
    1=&t_{n+1,1}+t_{n+1,2}+\ldots +t_{n+1,n}\frac{\lambda_n^2-1}{\lambda_n^2},
\end{cases}
\end{align}
and $s_{n+1}\in \rbb$ such that
\begin{align*}
s_{n+1}^2
&=\frac{1-\lambda_{n+1}^2}{\lambda_{n+1}^2}-\big\| t_{n+1,1}h_1+\ldots +t_{n+1,n}h_n\big\|^2\\
&=\frac{1-\lambda_{n+1}^2}{\lambda_{n+1}^2}-\Big(t_{n+1,1}^2\frac{1-\lambda_1^2}{\lambda_1^2}+\ldots +t_{n+1,n}^2\frac{1-\lambda_{n}^2}{\lambda_n}-2\sum_{i\neq j}t_{n+1,i}t_{n+1,j}\Big)
\end{align*}
Using linear algebra arguments one can show that \eqref{neon1} is solvable if $\lambda_{n}\in(0,+\infty)$ is small enough when compared with weights in $\{\lambda_1,\ldots,\lambda_{n-1}\}$. At the same time, $s_{n+1}$ is well defined if $\lambda_{n+1}\in(0,+\infty)$ is small enough when compared with weights in $\{\lambda_1,\ldots\lambda_{n}\}$. Therefore, the induction step $n\to n+1$ can be done as long as $\{\lambda_1,\ldots, \lambda_{n+1}\}$ are decreasing fast enough to 0.

Summing up all the above we see that $\slam$ can be shown to satisfy $C\slam^*=\slam C$ on $\lin\{e_j\colon j\in \zbb_+\}$, whenever $\lambdab$ consists of weights decreasing to 0 sufficiently fast. If this is the case, we also have $C\slam^*=\slam C$ because $\slam$ is a bounded operator on $\ell^2(V)$.
\end{exa}
Considering complex symmetric weighted shifts over infinite directed trees of finite depth may be interesting for another reason. In \cite{b-b-2023} we posed the problem of whether there exists an unbounded complex selfadjoint operator $T$ such that the domain of $T^2$ is trivial (let us recall that a closed densely defined operator $T$ is complex selfadjoint if $T=CT^*C$ with some conjugation $C$). As it turns out, the latter property is easy to achieve if one considers $T$ to be a weighted shift on a directed tree. At the same time, one still has a lot of room for manipulating weights to ensure other properties. This was used in the past e.g., for constructing a hyponormal operator whose square had a trivial domain (see \cite{j-j-s-pams-2014}). Thus one may consider weighted shifts on directed trees as a promising tool for solving the above mentioned problem. This however is definitely not a straightforward matter as shown in the next example. Indeed, we consider the simplest broom-like tree $\tcal$ that admits a weighted shift $\slam$ on $\tcal$ such that $\dz{\slam^2}$ is not dense in the underlying $\ell^2$-space and we show that such a weighted shift cannot be complex selfadjoint. The problem of wheter such a $\slam$ can be complex symmetric (i.e., $\slam\subseteq C\slam^*C$) remains open.
\begin{exa}
Let $V=\{0\}\cup\big(\{1,2\}\times\nbb\big)$ and $E=\big\{(0,(1,j))\colon j\in\nbb\big\}\cup \big\{ ((1,j),(2,j))\colon j\in\nbb\big\}$. Let $\lambdab=\{\lambda_v\}_{v\in V^\circ}\subseteq (0,\infty)$ satisfy 
\begin{align}\label{tt1}
\text{$\sum_{j\in\nbb}\lambda_{1,j}^2<\infty$\quad and\quad $\sum_{j\in\nbb}\lambda_{1,j}^2\lambda_{2,j}^2=\infty$.}
\end{align}
Let $\slam$ be a weighted shift on a directed tree $\tcal=(E,V)$ induced by $\lambdab$. Then, by \eqref{tt1}, $\slam$ is a densely defined operator such that $e_0\notin\dz{\slam^2}$. We have
\begin{align}\label{tt2}
\jd{\slam}^\perp=\cbb e_0\oplus \big(\hh_1\ominus \cbb f_1\big)\oplus \cbb f_1,\quad 
\jd{\slam^*}^\perp=\cbb f_1\oplus \hh_2
\end{align}
and
\begin{align}\label{tt3}
\jd{\slam}=\hh_2,\quad
\jd{\slam^*}=\cbb e_0 \oplus \big(\hh_1\ominus \cbb f_1\big),
\end{align}
where $\hh_1=\bigvee\{e_{1,j}\colon j\in\nbb\}$, $f_1=\frac{\sum_{i\in\nbb}\lambda_{1,i}e_{1,i}}{\big(\sum_{i\in\nbb}\lambda_{1,i}^2\big)^{1/2}}$, and $\hh_2=\bigvee\{e_{2,j}\colon j\in\nbb\}$. Assuming $\slam$ is complex selfadjoint with respect to a conjugation $C$, we deduce from \eqref{tt2} and \eqref{tt3} that $C f_1=\alpha f_1$, with $|\alpha|=1$. This however leads to a contradiction since
$\|C\slam^* f_1\|=\big(\sum_{j\in\nbb}\lambda_{1,j}^2\big)^{1/2}$ but $Cf_1\notin \dz{\slam}$ due to \eqref{tt1}.
\end{exa}
\section*{Statesments and declarations}
The author was supported by the Ministry of Science and Higher Education of the Republic of Poland. The author declare no competing interests. 

\end{document}